\documentclass[pdflatex,sn-mathphys-num]{sn-jnl}

\usepackage{graphicx}%
\usepackage{multirow}%
\usepackage{amsmath,amssymb,amsfonts}%
\usepackage{amsthm}%
\usepackage{mathrsfs}%
\usepackage[title]{appendix}%
\usepackage{xcolor}%
\usepackage{textcomp}%
\usepackage{manyfoot}%
\usepackage{booktabs}%
\usepackage{algorithm}%
\usepackage{algorithmicx}%
\usepackage{algpseudocode}%
\usepackage{listings}%
\usepackage{CJKutf8}
\usepackage{lineno}
\usepackage{float}
\usepackage{pdfpages}
\usepackage{caption}
\theoremstyle{thmstyleone}%
\newtheorem{theorem}{Theorem}
\theoremstyle{thmstyletwo}%

\theoremstyle{thmstylethree}%

\begin{document}

 	\title[title name]{A Non-compact Positivity-Preserving Numerical Scheme for   Elliptic Differential Equations Based on Mathematical Expectation}
	\author[1]{\fnm{Haoran} \sur{Xu}}\email{20244007005@stu.suda.edu.cn}

\author[1]{\fnm{Kunyang} \sur{Li}}\email{hfyuglu@163.com}

\author*[2]{\fnm{Xingye} \sur{Yue}}\email{xyyue@suda.edu.cn}

\affil*[1]{\orgdiv{School of Mathematical Sciences}, \orgname{Soochow University}, \orgaddress{\city{Suzhou}, \postcode{215006}, \country{China}}}

\affil[2]{\orgdiv{Center for Financial Engineering , School of Mathematical Sciences}, \orgname{Soochow University}, \orgaddress{\city{Suzhou}, \postcode{215006}, \country{China}}}


\abstract{
	We propose a novel non-compact, positivity-preserving scheme for linear non-divergence form elliptic equations. Based on the Feynman--Kac formula, the solution is represented as a conditional expectation associated with a diffusion process.Instead of using compact Markov chain approximations, we construct a wide-stencil scheme by approximating the expectation with carefully designed transition probabilities, ensuring both consistency and positivity preservation. The method is effective for anisotropic diffusion problems with mixed derivatives, where classical schemes typically fail unless the covariance matrix is diagonally dominant.
	
	A key feature of the proposed framework is its robust treatment of boundary conditions. For Dirichlet boundaries, we introduce a quadtree-based non-uniform stopping-time strategy, achieving $O(h)$ accuracy. For Neumann boundaries, a discrete specular reflection mechanism is employed, yielding $O(h^{1/2})$ convergence. Periodic boundaries are handled through modular wrapping, also achieving $O(h)$ accuracy.
	
	The resulting schemes are unconditionally stable and positivity-preserving due to their probabilistic structure. Numerical experiments confirm the theoretical convergence rates under all boundary conditions considered.
}




\maketitle

\section{Introduction}

Elliptic partial differential equations play a fundamental role in modeling
steady-state phenomena in physics, engineering, and finance. In many
applications, the quantities of interest—such as mass densities,
probability distributions, and option prices—are intrinsically
nonnegative. Preserving positivity in numerical approximations is therefore
essential for physical consistency, numerical stability, and computational
reliability. Positivity is closely related to several fundamental
properties, including monotonicity, the comparison principle, and the
maximum principle. For linear problems, positivity and monotonicity are
equivalent, while for nonlinear problems, positivity is a necessary but not
sufficient condition for monotonicity. At the discrete level, monotone
schemes typically ensure stability in the maximum norm. For fully nonlinear
equations, the theory of viscosity solutions provides the appropriate
analytical framework, and a classical result states that consistent,
stable, and monotone schemes converge to the unique viscosity solution
\cite{bib1}.

In one spatial dimension, standard discretization methods, such as finite
differences, finite elements, and finite volumes, can often be adjusted to
preserve positivity through suitable mesh constraints. In multiple
dimensions, however, the situation becomes significantly more challenging,
particularly for anisotropic diffusion problems involving mixed second-order
derivatives. It is well known that, for general two-dimensional linear
anisotropic problems, no linear compact nine-point scheme can simultaneously
achieve consistency and positivity preservation unless the diffusion matrix
is diagonally dominant \cite{bib11,bib3,bib13,bib12}. This limitation has
motivated two main research directions.

The first approach is based on nonlinear compact schemes, where nonlinearity
is introduced to enforce positivity while maintaining a compact stencil.
Such methods have been developed in the finite difference, finite element,
and finite volume settings \cite{bib13,bib10,bib14,bib17,bib18,bib7,bib20,bib21}.
Although effective, these methods often involve increased computational
complexity and may require problem-dependent design.

The second approach is to construct non-compact linear schemes. By relaxing
the compactness requirement, one can design linear schemes that are both
consistent and positivity-preserving. Representative examples include the
BZ method \cite{bib2} and semi-Lagrangian methods \cite{bib4}. These methods
typically decompose the diffusion operator into directional derivatives
along non-grid-aligned directions, thereby eliminating mixed derivative
terms. However, this leads to wide stencils and requires interpolation at
off-grid points. A major challenge in this framework is the treatment of
boundary conditions, as the extended stencils frequently exit the
computational domain, often leading to extrapolation and potential loss of
accuracy and monotonicity. While specialized techniques, such as local
coordinate transformations, can alleviate these issues in certain
applications \cite{bib24}, such approaches are generally problem-specific
and lack robustness.

In addition, compact positivity-preserving schemes have been developed for
specific classes of problems. For instance, Xu et al.\ \cite{bib16} proposed
a linear compact scheme for degenerate Fokker--Planck equations by
exploiting structural properties of the equation. However, such approaches
are not applicable to general anisotropic diffusion problems without
diagonal dominance.

An alternative perspective is provided by probabilistic methods. The
Feynman--Kac formula represents the solution of a PDE as a conditional
expectation associated with a diffusion process, which naturally preserves
positivity. Classical tree methods in computational finance \cite{bib5} and
Markov chain approximations for stochastic control \cite{bib9} are typical
examples. However, these methods often lead to compact discretizations that
inherit the same diagonal dominance restrictions in multi-dimensional
anisotropic settings. Other works have employed probabilistic
representations to solve high-dimensional nonlinear PDEs
\cite{bib6,bib8,bib22}, but their focus is typically different from the
construction of monotone finite difference schemes.

\medskip
\noindent\textbf{Our contribution.}
In this work, we develop a novel non-compact, positivity-preserving
numerical framework for linear elliptic equations in non-divergence form.
Our approach is inspired by probabilistic representations based on the
Feynman--Kac formula and extends our previous work on parabolic problems
\cite{bib23,bib25} to the elliptic setting.

Instead of directly discretizing the first-exit-time representation, we
approximate the elliptic solution via the long-time limit of an associated
parabolic problem. Based on this interpretation, we construct a
time-marching scheme over short time intervals and approximate the
resulting conditional expectation using a wide-stencil discretization with
carefully designed transition probabilities. These probabilities are
constructed to satisfy moment-matching conditions, ensuring consistency
with the underlying diffusion process while preserving positivity.

A key feature of the proposed framework is its unified and robust treatment
of boundary conditions. For Dirichlet boundaries, we introduce a
quadtree-based non-uniform stopping-time strategy that achieves an
$L^\infty$ error of order $O(h)$. For Neumann boundaries, we employ a
discrete specular reflection mechanism, leading to an $O(h^{1/2})$
convergence rate. Periodic boundary conditions are handled through modular
wrapping of the diffusion process, yielding $O(h)$ accuracy. In all cases,
the resulting schemes are unconditionally stable and positivity-preserving
as a direct consequence of their probabilistic structure.

The remainder of the paper is organized as follows. In
Section~\ref{sec2}, we present the numerical schemes and their
implementation for different types of boundary conditions, together with
theoretical analysis of consistency, stability, and positivity preservation.
Section~\ref{sec:numerical-results} provides numerical experiments that
validate the convergence rates and demonstrate the effectiveness of the
proposed methods. Finally, Section~\ref{sec4} concludes the paper.
\section{Numerical Scheme}
\label{sec2}
\subsection{Numerical Scheme for Boundary Conditions of the First Type}

Elliptic equations admit a probabilistic interpretation through the
first-exit-time representation, which naturally leads to Monte Carlo-type
algorithms. While such representations are well suited for high-dimensional
problems, they are generally not appropriate for constructing accurate
deterministic schemes in low dimensions.

Following the approach developed in \cite{bib25}, we instead approximate
the elliptic problem via the long-time limit of an associated parabolic
equation. Specifically, we consider the corresponding time-dependent
problem and construct a time-marching scheme based on the Feynman--Kac
representation over short time intervals. This approach avoids the direct
discretization of the exit-time formulation and leads to a tractable
local numerical scheme. This interpretation is valid under standard
assumptions ensuring the convergence of the parabolic solution to the
elliptic steady state.

The resulting probabilistic formulation provides the foundation for the
positivity-preserving wide-stencil discretization introduced below.

\subsubsection{Parabolic Partial Differential Equations}

Consider the backward parabolic problem:

\begin{equation}
	\label{eq:parabolic}
	\begin{cases}
		\partial_t f + \dfrac{1}{2} \mathrm{Tr}\left(AA^\top \nabla^2 f\right) + B \cdot \nabla f - r(x, y, t) f = 0, & (x, y) \in \Omega, \quad t \in [0, T), \\
		f|_{t = T} = \varphi(x, y), \\
		f|_{\partial \Omega} = f_\partial(x, y, t)|_{\partial \Omega},
	\end{cases}
\end{equation}

where \(\Omega = (x_0, x_{M_1}) \times (y_0, y_{M_2})\), 
\[
AA^\top = \begin{pmatrix} 
	\sigma_1^2 & \rho \sigma_1 \sigma_2 \\ 
	\rho \sigma_1 \sigma_2 & \sigma_2^2 
\end{pmatrix}, \quad 
B = (b_1, b_2),
\]
and the coefficients \(\rho, \sigma_1, \sigma_2, r, b_1, b_2\) are continuous in \((x, y, t)\), with \(\rho \in [-1, 1]\) and \(r \geq 0\) for all \((x, y) \in \Omega\), \(t \in [0, T]\). The Hessian and gradient are defined as
\[
D^2 f = \begin{pmatrix} 
	\partial_{xx} f & \partial_{xy} f \\ 
	\partial_{yx} f & \partial_{yy} f 
\end{pmatrix}, \quad 
\nabla f = \begin{pmatrix} 
	\partial_x f \\ 
	\partial_y f 
\end{pmatrix}.
\]

A particular choice for \(A\) is
\[
A = \begin{pmatrix} 
	\sigma_1 \cos\theta & \sigma_1 \sin\theta \\ 
	\sigma_2 \sin\theta & \sigma_2 \cos\theta 
\end{pmatrix}, \quad 
\theta = \frac{\arcsin \rho}{2} \in \left[-\frac{\pi}{4}, \frac{\pi}{4}\right].
\]

A time reversal \(\bar{t} = T - t\) converts \eqref{eq:parabolic} into a forward initial--boundary value problem.

We discretize the domain \(\Omega \times [0, T]\) using a spatial grid with steps \(h_1, h_2\) and a time step \(\Delta t\):
\begin{align*}
	x_i &= x_0 + i h_1, \quad i = 0, \dots, M_1, \quad h_1 = \frac{x_{M_1} - x_0}{M_1}, \\
	y_j &= y_0 + j h_2, \quad j = 0, \dots, M_2, \quad h_2 = \frac{y_{M_2} - y_0}{M_2}, \\
	t_n &= n \Delta t, \quad n = 0, \dots, N, \quad \Delta t = \frac{T}{N}.
\end{align*}
Let \(f_{h,i,j}^n \approx f(x_i, y_j, t_n)\) denote the numerical solution.

From the Feynman--Kac formula, for \(t_n \leq t \leq t_{n+1}\),
\begin{equation}
	\label{eq:feynman-kac}
	f(x_i, y_j, t_n) = \mathbb{E} \left[ \exp\left( -\int_{t_n}^{\tau} r(X_s, s) \, ds \right) u(X_\tau, Y_\tau, \tau) \,\middle|\, X_n = x_i, Y_n = y_j \right],
\end{equation}
where \((X_s, Y_s)\) satisfies
\begin{equation}
	\label{eq:sde}
	\begin{cases}
		dX(s) = b_1 ds + \sigma_1 \cos\theta \, dW_1(s) + \sigma_1 \sin\theta \, dW_2(s), \\
		dY(s) = b_2 ds + \sigma_2 \sin\theta \, dW_1(s) + \sigma_2 \cos\theta \, dW_2(s),
	\end{cases}
	\quad t_n < s \leq t_{n+1},
\end{equation}
with \(X_n = x_i\), \(Y_n = y_j\), and \(W_1, W_2\) being independent Brownian motions. The stopping time is
\[
\tau = t_{n+1} \wedge \inf \{ s : (X_s, Y_s) \notin \Omega \},
\]
and the payoff function is
\[
u(X_\tau, Y_\tau, \tau) = 
\begin{cases}
	f(X_{n+1}, Y_{n+1}, t_{n+1}), & \tau = t_{n+1}, \\
	f_\partial(X_\tau, Y_\tau, \tau), & \tau < t_{n+1}.
\end{cases}
\]

A discrete approximation of \((X_t, Y_t)\) over \([t_n, t]\) is given by:
\begin{align*}
	X_t &\approx x_i + b_1(x_i, y_j, t_n) \sqrt{t - t_n} + \sigma_1 \cos\theta \cdot \sqrt{t - t_n} \, \xi_1 + \sigma_1 \sin\theta \cdot \sqrt{t - t_n} \, \xi_2, \\
	Y_t &\approx y_j + b_2(x_i, y_j, t_n) \sqrt{t - t_n} + \sigma_2 \sin\theta \cdot \sqrt{t - t_n} \, \xi_1 + \sigma_2 \cos\theta \cdot \sqrt{t - t_n} \, \xi_2,
\end{align*}
where \(\xi_1, \xi_2 \sim \mathcal{N}(0,1)\). These are approximated by discrete random variables:
\[
\xi_1^h, \xi_2^h = 
\begin{cases}
	-1, & \text{probability } 1/2, \\
	+1, & \text{probability } 1/2,
\end{cases}
\]
matching the first three moments of the standard normal distribution.

Introduce the notation:
\begin{align*}
	b_{k,i,j}^n &= b_k(x_i, y_j, t_n), \quad 
	\sigma_{k,i,j}^n = \sigma_k(x_i, y_j, t_n), \quad 
	\theta_{i,j}^n = \theta(x_i, y_j, t_n), \\
	\alpha_{i,j}^n &= \sin\theta_{i,j}^n + \cos\theta_{i,j}^n, \quad 
	\beta_{i,j}^n = \sin\theta_{i,j}^n - \cos\theta_{i,j}^n, \\
	r_{i,j}^n &= r(x_i, y_j, t_n), \quad 
	\rho_{i,j}^n = \rho(x_i, y_j, t_n).
\end{align*}
Note that \(\alpha_{i,j}^n \in [0, \sqrt{2}]\), \(\beta_{i,j}^n \in [-\sqrt{2}, 0]\), and
\[
(\alpha_{i,j}^n)^2 + (\beta_{i,j}^n)^2 = 2, \quad
(\alpha_{i,j}^n)^2 - (\beta_{i,j}^n)^2 = 2\rho_{i,j}^n.
\]

The discrete process \((X_t^h, Y_t^h)\) follows four equiprobable branches:
\[
(X_t^h,Y_t^h)=
\begin{cases}
	\big(x_i + b_{1,i,j}^n (t-t_n) + \alpha_{i,j}^n\sigma_{1,i,j}^n\sqrt{t-t_n},\ 
	y_j + b_{2,i,j}^n (t-t_n) + \alpha_{i,j}^n\sigma_{2,i,j}^n\sqrt{t-t_n}\big),\\[4pt]
	\big(x_i + b_{1,i,j}^n (t-t_n) - \beta_{i,j}^n\sigma_{1,i,j}^n\sqrt{t-t_n},\ 
	y_j + b_{2,i,j}^n (t-t_n) + \beta_{i,j}^n\sigma_{2,i,j}^n\sqrt{t-t_n}\big),\\[4pt]
	\big(x_i + b_{1,i,j}^n (t-t_n) - \alpha_{i,j}^n\sigma_{1,i,j}^n\sqrt{t-t_n},\ 
	y_j + b_{2,i,j}^n (t-t_n) - \alpha_{i,j}^n\sigma_{2,i,j}^n\sqrt{t-t_n}\big),\\[4pt]
	\big(x_i + b_{1,i,j}^n (t-t_n) + \beta_{i,j}^n\sigma_{1,i,j}^n\sqrt{t-t_n},\ 
	y_j + b_{2,i,j}^n (t-t_n) - \beta_{i,j}^n\sigma_{2,i,j}^n\sqrt{t-t_n}\big).
\end{cases}
\]
where \(\Delta t = t - t_n\), and coefficient evaluations at \((x_i, y_j, t_n)\) are omitted for brevity.

The discrete stopping time is
\[
\tau^h = t_{n+1} \wedge \inf \left\{ s : (X_s^h, Y_s^h) \notin \Omega \right\}.
\]
Let \(\hat{\tau}_k = \tau_k - t_n\) for \(k = 1, \dots, 4\). A naive numerical scheme would be:
\begin{equation}\label{eq:naive-approx}
	f_{h,i,j}^n
	= \mathbb{E}^h\!\Big[\frac{u(X_{\tau^h}^{h},Y_{\tau^h}^h,\tau^h)}{1+r_{i,j}^n\hat\tau^h}\Big]
	= \sum_{k=1}^4\frac{1}{4(1+r_{i,j}^n\hat\tau_k)}\,u(X_{\tau_k}^{h,k},Y_{\tau_k}^{h,k},\tau_k),
\end{equation}
which may suffer boundary-induced loss of accuracy. To restore local consistency we introduce branch-dependent probabilities $\omega_k$ and adopt the scheme
\begin{equation}
	\label{eq:weighted-scheme}
	f_{h,i,j}^n = \sum_{k=1}^4 \frac{\omega_k}{1 + r_{i,j}^n \hat{\tau}_k} u(X_{\tau_k}^{h,k}, Y_{\tau_k}^{h,k}, \tau_k),
\end{equation}
where one convenient closed form for $\{\omega_k\}$ that enforces the required moment-matching conditions is
\begin{equation}
	\label{eq:weights}
	\begin{cases}
		\omega_1 = \dfrac{\sqrt{\hat{\tau}_2 \hat{\tau}_3 \hat{\tau}_4}}{\left(\sqrt{\hat{\tau}_1} + \sqrt{\hat{\tau}_3}\right)\left(\sqrt{\hat{\tau}_1 \hat{\tau}_3} + \sqrt{\hat{\tau}_2 \hat{\tau}_4}\right)}, \\
		\omega_2 = \dfrac{\sqrt{\hat{\tau}_1 \hat{\tau}_3 \hat{\tau}_4}}{\left(\sqrt{\hat{\tau}_2} + \sqrt{\hat{\tau}_4}\right)\left(\sqrt{\hat{\tau}_1 \hat{\tau}_3} + \sqrt{\hat{\tau}_2 \hat{\tau}_4}\right)}, \\
		\omega_3 = \dfrac{\sqrt{\hat{\tau}_1 \hat{\tau}_2 \hat{\tau}_4}}{\left(\sqrt{\hat{\tau}_1} + \sqrt{\hat{\tau}_3}\right)\left(\sqrt{\hat{\tau}_1 \hat{\tau}_3} + \sqrt{\hat{\tau}_2 \hat{\tau}_4}\right)}, \\
		\omega_4 = \dfrac{\sqrt{\hat{\tau}_1 \hat{\tau}_2 \hat{\tau}_3}}{\left(\sqrt{\hat{\tau}_2} + \sqrt{\hat{\tau}_4}\right)\left(\sqrt{\hat{\tau}_1 \hat{\tau}_3} + \sqrt{\hat{\tau}_2 \hat{\tau}_4}\right)}.
	\end{cases}
\end{equation}
A direct algebraic check shows $\omega_k\ge0$ and $\sum_{k=1}^4\omega_k=1$ ; combined with a positivity-preserving interpolation for off-grid evaluations of $u$, the scheme preserves positivity. Since the scheme may require values at non-grid points at \(t_{n+1}\), we introduce a positivity-preserving bilinear interpolation operator \(\mathcal{L}^{(1)}\). The final numerical scheme becomes:
\begin{equation}
	\label{eq:final-scheme}
	f_{h,i,j}^n = \sum_{k=1}^4 \frac{\omega_k}{1 + r_{i,j}^n \hat{\tau}_k} U(X_{\tau_k}^{h,k}, Y_{\tau_k}^{h,k}, \tau_k),
\end{equation}
where
\begin{equation}
	U(X_{\tau_k}^{h,k}, Y_{\tau_k}^{h,k}, \tau_k) = 
	\begin{cases}
		\mathcal{L}^{(1)} f_h(X_{\tau_k}^{h,k}, Y_{\tau_k}^{h,k}, t_{n+1}), & (X_{\tau_k}^{h,k}, Y_{\tau_k}^{h,k}) \in \Omega, \\
		f_\partial(X_{\tau_k}^{h,k}, Y_{\tau_k}^{h,k}, \tau_k), & (X_{\tau_k}^{h,k}, Y_{\tau_k}^{h,k}) \in \partial \Omega.
	\end{cases}
\end{equation}

If the stopping occurs at \(t_{n+1}\) without boundary contact, \(U\) is computed via linear interpolation from the four nearest spatial grid points.

\subsubsection{Elliptic Partial Differential Equations}

Consider the elliptic partial differential equation:
\begin{equation}
	\label{eq:elliptic}
	\begin{cases}
		\dfrac{1}{2} \mathrm{Tr}\left(AA^\top D^2f\right) + B \cdot \nabla f - r(x,y)f + q(x,y) = 0, & (x,y) \in \Omega, \\
		f|_{\partial\Omega} = f_{\partial}(x,y),
	\end{cases}
\end{equation}

where the domain is $\Omega = (x_0, x_{M_1}) \times (y_0, y_{M_2})$, with coefficient matrices
\[
AA^\top = \begin{pmatrix} 
	\sigma_1^2 & \rho \sigma_1 \sigma_2 \\ 
	\rho \sigma_1 \sigma_2 & \sigma_2^2 
\end{pmatrix}, \quad 
B = (b_1, b_2),
\]
and the coefficients $\rho, \sigma_1, \sigma_2, r, b_1, b_2$ are continuous functions of $(x,y)$ satisfying
\[
\inf_{(x,y) \in \Omega} \sigma_1^2 > 0, \quad \inf_{(x,y) \in \Omega} \sigma_2^2 > 0, \quad \rho \in [-1, 1], \quad r \geq 0 \quad \forall (x,y) \in \Omega.
\]

we choose the parametrization:
\[
A = \begin{pmatrix} 
	\sigma_1 \cos\theta & \sigma_1 \sin\theta \\ 
	\sigma_2 \sin\theta & \sigma_2 \cos\theta 
\end{pmatrix}, \quad 
\theta = \frac{\arcsin \rho}{2} \in \left[-\frac{\pi}{4}, \frac{\pi}{4}\right].
\]
The differential operators are defined as $D^2f = \begin{pmatrix} 
	\partial_{xx} f & \partial_{xy} f \\ 
	\partial_{yx} f & \partial_{yy} f 
\end{pmatrix}$ and $\nabla f = \begin{pmatrix} 
	\partial_x f \\ 
	\partial_y f 
\end{pmatrix}$.

We discretize the spatial domain using uniform grids:
\begin{align*}
	x_i &= x_0 + i \cdot h_1, \quad 0 \leq i \leq M_1, \quad h_1 = \frac{x_{M_1} - x_0}{M_1}, \\
	y_j &= y_0 + j \cdot h_2, \quad 0 \leq j \leq M_2, \quad h_2 = \frac{y_{M_2} - y_0}{M_2}.
\end{align*}
The numerical solution at grid point $(x_i, y_j)$ is denoted by $f_{h,i,j}$.

To ensure the well-posedness of our numerical scheme, we assume there exists a constant $r_0 > 0$ such that $r(x,y) > r_0$ for all $(x,y) \in \Omega$. This assumption is justified by the following transformation: if $r \geq 0$ and vanishes at some points, we define $f(x,y) = (2 - e^{-a(x - x_0)})w(x,y)$ and substitute into equation (\ref{eq:elliptic}) to obtain:
\begin{equation*}
	\begin{cases}
		\dfrac{1}{2} \mathrm{Tr}(\widetilde{A} \widetilde{A}^\top D^2w) + \widetilde{B} \cdot \nabla w - \widetilde{r}(x,y)w + q(x,y) = 0, & (x,y) \in \Omega, \\
		w|_{\partial\Omega} = \dfrac{f_{\partial}(x,y)}{2 - e^{-a(x - x_0)}},
	\end{cases}
\end{equation*}
where
\begin{align*}
	\widetilde{r} &= \dfrac{\sigma_1^2}{2} a^2 e^{-a(x - x_0)} - b_1 a e^{-a(x - x_0)} + (2 - e^{-a(x - x_0)})r, \\
	\widetilde{A} \widetilde{A}^\top &= \begin{pmatrix}
		\sigma_1^2 (2 - e^{-a(x - x_0)}) & \rho \sigma_1 \sigma_2 (2 - e^{-a(x - x_0)}) \\
		\rho \sigma_1 \sigma_2 (2 - e^{-a(x - x_0)}) & \sigma_2^2 (2 - e^{-a(x - x_0)})
	\end{pmatrix}, \\
	\widetilde{B} &= \begin{pmatrix}
		\sigma_1^2 a e^{-a(x - x_0)} + (2 - e^{-a(x - x_0)})b_1 \\
		(2 - e^{-a(x - x_0)})b_2 + \rho \sigma_1 \sigma_2 a e^{-a(x - x_0)}
	\end{pmatrix}.
\end{align*}
By choosing $a = \dfrac{2(\sup_{(x,y) \in \Omega} |b(x,y)| + 1)}{\inf_{(x,y) \in \Omega} \sigma_1^2}$, we ensure $\inf_{(x,y) \in \Omega} \widetilde{r} > 0$.

Since elliptic equations represent steady states of parabolic equations, we note that $\partial_t f = 0$ and for any $T > 0$, $f|_{t = T} = f(x,y)$, where $f(x,y)$ solves equation (\ref{eq:elliptic}). This leads to the parabolic formulation:
\begin{equation}
	\label{eq:parabolic-reformulation}
	\begin{cases}
		\partial_t f + \dfrac{1}{2} \mathrm{Tr}(AA^\top D^2f) + B \cdot \nabla f - r(x,y)f + q(x,y) = 0, & (x,y) \in \Omega, \\
		f|_{t = T} = f(x,y), \\
		f|_{\partial\Omega} = f_{\partial}(x,y).
	\end{cases}
\end{equation}

Adapting the parabolic scheme to the spatial context, we derive the numerical scheme for elliptic equations:
\begin{equation} 
	\label{eq:elliptic-scheme}
	f_{h,i,j} = \sum_{k=1}^4 \left( \dfrac{\omega_k}{1 + r_{i,j}{s}_k} U\left(X_{s_k}^{h,k}, Y_{s_k}^{h,k}\right) + q(x_i, y_j) \cdot \omega_k {s}_k \right),
\end{equation}
where the utility function is defined as:
\begin{equation*}
	U\left(X_{s_k}^{h,k}, Y_{s_k}^{h,k}\right) = 
	\begin{cases}
		\mathcal{L}^{(1)} f_h\left(X_{h}^{h,k}, Y_{h}^{h,k}\right), & \left(X_{h}^{h,k}, Y_{h}^{h,k}\right) \in \Omega, \\
		f_{\partial}\left(X_{s_k}^{h,k}, Y_{s_k}^{h,k}\right), & \left(X_{s_k}^{h,k}, Y_{s_k}^{h,k}\right) \in \partial \Omega.
	\end{cases}
\end{equation*}

The  hitting time ${s}_k$ are defined as follows:
\begin{equation}
	\label{eq:exit-times}
	\begin{cases}
		{s}_1 = h \wedge \inf \left\{ s : \left( x_i + b_{1,i,j} s + \alpha_{i,j} \sigma_{1,i,j} \sqrt{s}, y_j + b_{2,i,j} s + \alpha_{i,j} \sigma_{2,i,j} \sqrt{s} \right) \notin \Omega, 0 \leq s \leq h \right\}, \\
		{s}_2 = h \wedge \inf \left\{ s : \left( x_i + b_{1,i,j} s - \beta_{i,j} \sigma_{1,i,j} \sqrt{s}, y_j + b_{2,i,j} s + \beta_{i,j} \sigma_{2,i,j} \sqrt{s} \right) \notin \Omega, 0 \leq s \leq h \right\}, \\
		{s}_3 = h \wedge \inf \left\{ s : \left( x_i + b_{1,i,j} s - \alpha_{i,j} \sigma_{1,i,j} \sqrt{s}, y_j + b_{2,i,j} s - \alpha_{i,j} \sigma_{2,i,j} \sqrt{s} \right) \notin \Omega, 0 \leq s \leq h \right\}, \\
		{s}_4 = h \wedge \inf \left\{ s : \left( x_i + b_{1,i,j} s + \beta_{i,j} \sigma_{1,i,j} \sqrt{s}, y_j + b_{2,i,j} s - \beta_{i,j} \sigma_{2,i,j} \sqrt{s} \right) \notin \Omega, 0 \leq s \leq h \right\},
	\end{cases}
\end{equation}
with $h = \min(h_1, h_2)$.

The corresponding four discrete spatial branches are:
\begin{equation}
	\label{eq:spatial-branches}
	\begin{cases}
		\left(X_{s_1}^{h,k}, Y_{s_1}^{h,k}\right) = \left( x_i + b_{1,i,j} {s}_1 + \alpha_{i,j} \sigma_{1,i,j} \sqrt{{s}_1}, y_j + b_{2,i,j} {s}_1 + \alpha_{i,j} \sigma_{2,i,j} \sqrt{{s}_1} \right), \\
		\left(X_{s_2}^{h,k}, Y_{s_2}^{h,k}\right) = \left( x_i + b_{1,i,j} {s}_2 - \beta_{i,j} \sigma_{1,i,j} \sqrt{{s}_2}, y_j + b_{2,i,j}{s}_2 + \beta_{i,j} \sigma_{2,i,j} \sqrt{{s}_2} \right), \\
		\left(X_{s_3}^{h,k}, Y_{s_3}^{h,k}\right) = \left( x_i + b_{1,i,j} {s}_3 - \alpha_{i,j} \sigma_{1,i,j} \sqrt{{s}_3}, y_j + b_{2,i,j} {s}_3 - \alpha_{i,j} \sigma_{2,i,j} \sqrt{{s}_3} \right), \\
		\left(X_{s_4}^{h,k}, Y_{s_4}^{h,k}\right) = \left( x_i + b_{1,i,j} {s}_4 + \beta_{i,j} \sigma_{1,i,j} \sqrt{{s}_4}, y_j + b_{2,i,j} {s}_4 - \beta_{i,j} \sigma_{2,i,j} \sqrt{{s}_4} \right).
	\end{cases}
\end{equation}

The coefficients are evaluated spatially:
\begin{align*}
	b_{k,i,j} &= b_k(x_i, y_j), \quad 
	\sigma_{k,i,j} = \sigma_k(x_i, y_j), \quad 
	\theta_{i,j} = \theta(x_i, y_j), \\
	\alpha_{i,j} &= \sin(\theta_{i,j}) + \cos(\theta_{i,j}), \quad 
	\beta_{i,j} = \sin(\theta_{i,j}) - \cos(\theta_{i,j}), \\
	r_{i,j} &= r(x_i, y_j), \quad 
	\rho_{i,j} = \rho(x_i, y_j),
\end{align*}
with $\alpha_{i,j} \in [0, \sqrt{2}]$, $\beta_{i,j} \in [-\sqrt{2}, 0]$, and satisfying the identities:
\[
(\alpha_{i,j})^2 + (\beta_{i,j})^2 = 2, \quad (\alpha_{i,j})^2 - (\beta_{i,j})^2 = 2\rho_{i,j}.
\]

The branching probabilities are defined as:
\begin{equation}
	\label{eq:spatial-weights}
	\begin{cases}
		\omega_1 = \dfrac{\sqrt{{s}_2 {s}_3 {s}_4}}{\left(\sqrt{{s}_1} + \sqrt{{s}_3}\right)\left(\sqrt{{s}_1 {s}_3} + \sqrt{{s}_2 {s}_4}\right)}, \\
		
		\omega_2 = \dfrac{\sqrt{{s}_1{s}_3 {s}_4}}{\left(\sqrt{{s}_2} + \sqrt{{s}_4}\right)\left(\sqrt{{s}_1 {s}_3} + \sqrt{{s}_2 {s}_4}\right)}, \\
		
		\omega_3 = \dfrac{\sqrt{{s}_1 {s}_2 {s}_4}}{\left(\sqrt{{s}_1} + \sqrt{{s}_3}\right)\left(\sqrt{{s}_1 {s}_3} + \sqrt{{s}_2 {s}_4}\right)}, \\
		
		\omega_4 = \dfrac{\sqrt{{s}_1 {s}_2 {s}_3}}{\left(\sqrt{{s}_2} + \sqrt{{s}_4}\right)\left(\sqrt{{s}_1{s}_3} + \sqrt{{s}_2 {s}_4}\right)}.
	\end{cases}
\end{equation}

This constitutes an implicit scheme. The detailed algorithmic implementation is as follows:

\begin{algorithm}
	\caption{Implementation of Scheme (\ref{eq:elliptic-scheme})}
	\label{alg:scheme-elliptic-dirichlet}
	\begin{algorithmic}[1]
		\Require Interior grid points $(x_i, y_j)$, $i=1,\dots,M_1-1$, $j=1,\dots,M_2-1$
		\Ensure Coefficient matrix $T$ and right-hand side vector $b$
		
		\State Initialize $T \gets \mathbf{0}_{(M_1-1)(M_2-1) \times (M_1-1)(M_2-1)}$
		\State Initialize $b \gets \mathbf{0}_{(M_1-1)(M_2-1)}$
		
		\For{$i = 1$ to $M_1-1$}
		\For{$j = 1$ to $M_2-1$}
		\State $idx \gets (i-1)(M_2-1) + j$ \Comment{Global index}
		\State Compute  $s_k$ via \eqref{eq:exit-times}
		\State Compute $\omega_k$ via \eqref{eq:spatial-weights}
		\State Compute branch positions $(X_{s_k}^{h,k}, Y_{s_k}^{h,k})$ via \eqref{eq:spatial-branches}
		
		\For{$k = 1$ to $4$}
		\If{$(X_{s_k}^{h,k}, Y_{s_k}^{h,k}) \in \Omega$}
		\State Find interpolation stencil for $\mathcal{L}^{(1)} f_h$
		\For{each grid point $(i_m, j_n)$ in stencil}
		\State $idx_{mn} \gets i_m (M_2+1) + j_n$
		\State $l_{mn} \gets $ interpolation weight
		\State $T(idx, idx_{mn}) \gets T(idx, idx_{mn}) - \dfrac{\omega_k l_{mn}}{1 + r_{i,j} h}$
		\EndFor
		\Else{$(X_{s_k}^{h,k}, Y_{s_k}^{h,k}) \in \partial\Omega$}
		\State $b(idx) \gets b(idx) + \dfrac{\omega_k}{1 + r_{i,j} s_k} f_\partial(X_{s_k}^{h,k}, Y_{s_k}^{h,k})$
		\EndIf
		\EndFor
		
		\State $b(idx) \gets b(idx) + q(x_i, y_j) \sum_{k=1}^4 \omega_k s_k$ \Comment{Source term}
		\State $T(idx, idx) \gets T(idx, idx) + 1$ \Comment{Diagonal entry}
		\EndFor
		\EndFor
		
		\State Solve linear system $T F = b$ for $F$ \Comment{$F$ contains interior values}
	\end{algorithmic}
\end{algorithm}
\newpage
Due to the existence of boundary-hitting points and $r_{i,j} > 0$, the coefficient matrix $T$ satisfies:
\[
T_{i,i} > 0, \quad T_{i,j} \leq 0, \quad \text{and} \quad \sum_{j} T_{i,j} \geq 0,
\]
with at least one row $p$ satisfying $\sum_{j} T_{p,j} > 0$. Therefore, $T$ is an M-matrix, ensuring the positivity-preserving property of Scheme \eqref{eq:elliptic-scheme}.

Since $T$ is a sparse M-matrix, we can employ iterative methods such as Jacobi, Gauss-Seidel, or successive over-relaxation (SOR) to solve the linear system.
\begin{figure}[htbp]
	\centering
	\includegraphics[width=0.8\textwidth]{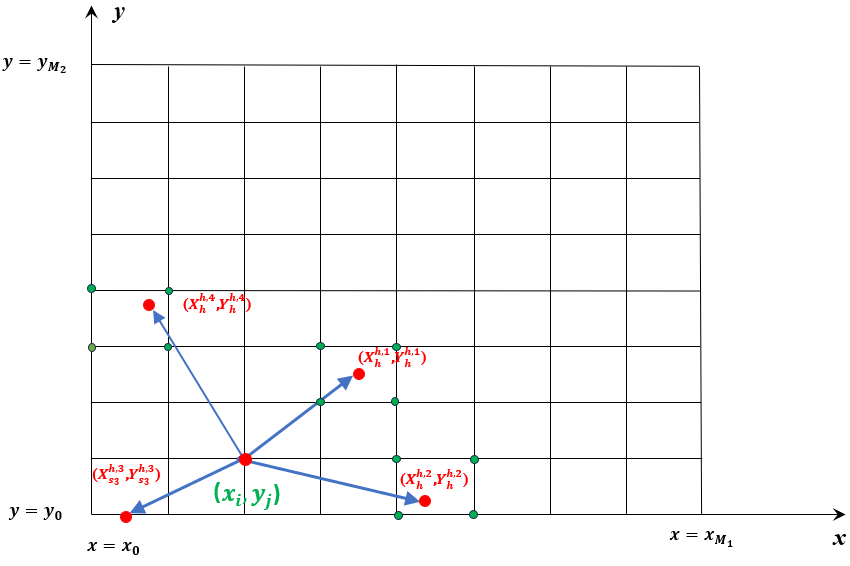}
	\caption{Illustration of the  Algorithm \ref{alg:scheme-elliptic-dirichlet} , where $s_1 = s_2 = s_4 = h$ and $s_3 <h$.}
	\label{fig2}
\end{figure}	
We now analyze the convergence of Algorithm \ref{alg:scheme-elliptic-dirichlet}. Define the maximum norm:
\[
\|f_{h,i,j}\|_\infty = \max_{0 \leq i \leq M_1} \max_{0 \leq j \leq M_2} |f_{h,i,j}|.
\]

\begin{theorem}
	\label{thm:convergence}
	Let $f$ be sufficiently smooth and consider the mesh scaling $h = h_1 = h_2$. 
	Define the pointwise error at grid node $(i,j)$ by
	\[
	e_{i,j} = f_{h,i,j} - f(x_i, y_j),
	\]
	Then the numerical solution produced by scheme~\eqref{alg:scheme-elliptic-dirichlet}satisfies
	\[
	\|e^n\|_\infty\leq O( h).
	\]
\end{theorem}

\begin{proof}
	The numerical error satisfies:
	\begin{align*}
		e_{i,j} &= \sum_{k=1}^4 \left( \dfrac{\omega_k}{1 + r_{i,j} s_k} U\left(X_{s_k}^{h,k}, Y_{s_k}^{h,k}\right) + q(x_i, y_j) \cdot \omega_k s_k \right) - f(x_i, y_j) \\
		&= \sum_{(X_{h}^{h,k}, Y_{h}^{h,k}) \in \Omega} \dfrac{\omega_k}{1 + r_{i,j} h} \mathcal{L}^{(1)} (f_h-f)\left(X_{h}^{h,k}, Y_{h}^{h,k}\right) \\
		&\quad + \sum_{(X_{h}^{h,k}, Y_{h}^{h,k}) \in \Omega} \dfrac{\omega_k}{1 + r_{i,j} h} (\mathcal{L}^{(1)} f - f)\left(X_{h}^{h,k}, Y_{h}^{h,k}\right) \\
		&\quad + \sum_{k=1}^4 \left( \dfrac{\omega_k}{1 + r_{i,j} s_k} f\left(X_{s_k}^{h,k}, Y_{s_k}^{h,k}\right) + q(x_i, y_j) \cdot \omega_k s_k - f(x_i, y_j)\right) .
	\end{align*}
	Thus, the error at any grid point can be expressed as:
	\[
	\begin{aligned}
		e_{i,j} = & \sum_{k=1}^4 \dfrac{\omega_k}{1 + r_{i,j} s_k} e\left(X_{s_k}^{h,k}, Y_{s_k}^{h,k}\right) + I_1 + O\left(h^2\right),
	\end{aligned}
	\]
	where
	\begin{align*}
		I_1 &= \omega_1 (1 - r_{i,j} s_1) f\left(x_i + b_{1,i,j} s_1 + \alpha_{i,j} \sigma_{1,i,j} \sqrt{s_1}, y_j + b_{2,i,j} s_1 + \alpha_{i,j} \sigma_{2,i,j} \sqrt{s_1}\right) \\
		&\quad + \omega_2 (1 - r_{i,j} s_2) f\left(x_i + b_{1,i,j} s_2 - \beta_{i,j} \sigma_{1,i,j} \sqrt{s_2}, y_j + b_{2,i,j} s_2 + \beta_{i,j} \sigma_{2,i,j} \sqrt{s_2}\right) \\
		&\quad + \omega_3 (1 - r_{i,j} s_3) f\left(x_i + b_{1,i,j} s_3 - \alpha_{i,j} \sigma_{1,i,j} \sqrt{s_3}, y_j + b_{2,i,j} s_3 - \alpha_{i,j} \sigma_{2,i,j} \sqrt{s_3}\right) \\
		&\quad + \omega_4 (1 - r_{i,j} s_4) f\left(x_i + b_{1,i,j} s_4 + \beta_{i,j} \sigma_{1,i,j} \sqrt{s_4}, y_j + b_{2,i,j} s_4 - \beta_{i,j} \sigma_{2,i,j} \sqrt{s_4}\right) \\
		&\quad + q(x_i, y_j) \cdot (\omega_1 s_1 + \omega_2 s_2 + \omega_3 s_3 + \omega_4 s_4) - f(x_i, y_j) \\
		&= \left[ A_2 \partial_{xx}^2 f + A_3 \partial_{xy}^2 f + A_4 \partial_{yy}^2 f + A_5 \partial_x f + A_6 \partial_y f + A_7 f + A_1 q \right] \bigg|_{x=x_i, y=y_j} + R,
	\end{align*}
	with coefficients:
	\begin{align*}
		A_1 &= \omega_1 s_1 + \omega_2 s_2 + \omega_3 s_3 + \omega_4 s_4, \\
		A_2 &= \frac{(\sigma_{1,i,j})^2}{2} A_1 + \frac{\rho_{i,j} (\sigma_{1,i,j})^2}{2} (\omega_1 s_1 - \omega_2 s_2 + \omega_3 s_3 - \omega_4 s_4), \\
		A_3 &= \rho_{i,j} \sigma_{1,i,j} \sigma_{2,i,j} A_1 + \sigma_{1,i,j} \sigma_{2,i,j} (\omega_1 s_1 - \omega_2 s_2 + \omega_3 s_3 - \omega_4 s_4), \\
		A_4 &= \frac{(\sigma_{2,i,j})^2}{2} A_1 + \frac{\rho_{i,j} (\sigma_{2,i,j})^2}{2} (\omega_1 s_1 - \omega_2 s_2 + \omega_3 s_3 - \omega_4 s_4), \\
		A_5 &= b_{1,i,j} A_1 + \omega_1 \alpha_{i,j} \sigma_{1,i,j} \sqrt{s_1} - \omega_2 \beta_{i,j} \sigma_{1,i,j} \sqrt{s_2} - \omega_3 \alpha_{i,j} \sigma_{1,i,j} \sqrt{s_3} + \omega_4 \beta_{i,j} \sigma_{1,i,j} \sqrt{s_4}, \\
		A_6 &= b_{2,i,j} A_1 + \omega_1 \alpha_{i,j} \sigma_{2,i,j} \sqrt{s_1} + \omega_2 \beta_{i,j} \sigma_{2,i,j} \sqrt{s_2} - \omega_3 \alpha_{i,j} \sigma_{2,i,j} \sqrt{s_3} - \omega_4 \beta_{i,j} \sigma_{2,i,j} \sqrt{s_4}, \\
		A_7 &= -r_{i,j} A_1.
	\end{align*}
	Here $R = \sum_{k=1}^4 O(\omega_k s_k^{3/2})$ denotes the remainder from the Taylor expansion.
	
	Using the consistency conditions:
	\begin{align*}
		&\omega_1 + \omega_2 + \omega_3 + \omega_4 = 1, \\
		&\omega_1 s_1 - \omega_2 s_2 + \omega_3 s_3 - \omega_4 s_4 = 0, \\
		&\omega_1 \alpha_{i,j} \sigma_{1,i,j} \sqrt{s_1} - \omega_2 \beta_{i,j} \sigma_{1,i,j} \sqrt{s_2} - \omega_3 \alpha_{i,j} \sigma_{1,i,j} \sqrt{s_3} + \omega_4 \beta_{i,j} \sigma_{1,i,j} \sqrt{s_4} = 0, \\
		&\omega_1 \alpha_{i,j} \sigma_{2,i,j} \sqrt{s_1} + \omega_2 \beta_{i,j} \sigma_{2,i,j} \sqrt{s_2} - \omega_3 \alpha_{i,j} \sigma_{2,i,j} \sqrt{s_3} - \omega_4 \beta_{i,j} \sigma_{2,i,j} \sqrt{s_4} = 0,
	\end{align*}
	we simplify the coefficients to:
	\begin{align*}
		A_2 &= \frac{(\sigma_{1,i,j})^2}{2} A_1, \quad 
		A_3 = \rho_{i,j} \sigma_{1,i,j} \sigma_{2,i,j} A_1, \quad 
		A_4 = \frac{(\sigma_{2,i,j})^2}{2} A_1, \\
		A_5 &= b_{1,i,j} A_1, \quad 
		A_6 = b_{2,i,j} A_1, \quad 
		A_7 = -r_{i,j} A_1.
	\end{align*}
	
	Therefore,
	\begin{align*}
		I_1 &= A_1 \left[ \frac{\sigma_1^2}{2} \partial_{xx}^2 f + \rho \sigma_1 \sigma_2 \partial_{xy}^2 f + \frac{\sigma_2^2}{2} \partial_{yy}^2 f + b_1 \partial_x f + b_2 \partial_y f - r f + q \right] \bigg|_{x=x_i, y=y_j} + R \\
		&= R \quad \text{(by equation (\ref{eq:elliptic}))}.
	\end{align*}
	
	Now we distinguish two cases for the step sizes $s_k$:
	\begin{itemize}
		\item \textbf{Case 1:} $s_k = h$ for all $k = 1,2,3,4$. Then due to the symmetry of the quadrature points, the remainder $R$ becomes $O(h^2)$ (the $O(h^{3/2})$ terms cancel, leaving a leading term of order $h^2$).
		\item \textbf{Case 2:} At least one $s_k < h$. In this situation the smallest step size determines the order, and we have $R = O(h^{3/2})$.
	\end{itemize}
	
	Substituting these estimates into the expression for $e_{i,j}$ yields
	\[
	e_{i,j} = \sum_{k=1}^4 \dfrac{\omega_k}{1 + r_{i,j} s_k} e\left(X_{s_k}^{h,k}, Y_{s_k}^{h,k}\right) + E,
	\]
	where
	\[
	E = \begin{cases}
		O(h^2), & \text{in Case 1},\\
		O(h^{3/2}), & \text{in Case 2}.
	\end{cases}
	\]
	
	Taking absolute values and using the maximum norm $\|e\|_\infty = \max_{i,j}|e_{i,j}|$, we obtain the following bounds:
	\begin{itemize}
		\item In Case 1, since $\sum_{k=1}^4 \omega_k = 1$ and $1/(1+r_{i,j}h) \le 1/(1+r_0 h)$ with $r_0 = \min_{i,j} r_{i,j} > 0$, we have
		\[
		|e_{i,j}| \le \frac{1}{1+r_0 h}\,\|e\|_\infty + C h^2.
		\]
		\item In Case 2, the coefficients $\omega_k/(1+r_{i,j}s_k)$ satisfy $\sum_{k=1}^4 \frac{\omega_k}{1+r_{i,j}s_k} \le c < 1$ (e.g., $c = 3/4$ can be taken under the given parameter choices). Hence
		\[
		|e_{i,j}| \le c\,\|e\|_\infty + C h^{3/2}.
		\]
	\end{itemize}
	
	Let $\|e\|_\infty$ be attained at some grid point. In Case 1 we obtain
	\[
	\|e\|_\infty \le \frac{1}{1+r_0 h}\,\|e\|_\infty + C h^2,
	\]
	which gives $\|e\|_\infty \le \frac{1+r_0 h}{r_0 h}\,C h^2 = O(h)$. In Case 2 we have
	\[
	\|e\|_\infty \le c\,\|e\|_\infty + C h^{3/2},
	\]
	so that $\|e\|_\infty \le \frac{C}{1-c}\,h^{3/2} = O(h^{3/2})$, which is also $O(h)$. Therefore, in both cases we conclude $\|e\|_\infty \le O(h)$, completing the proof.
\end{proof}
\subsection{Reflective Boundary Conditions }

Consider the elliptic partial differential equation with homogeneous Neumann boundary conditions:
\begin{equation}
	\label{eq:elliptic-neumann}
	\begin{cases}
		\dfrac{1}{2} \mathrm{Tr} \left( A A^\top D^2 f \right) + B \cdot \nabla f - r(x,y) f + q(x,y) = 0, & (x,y) \in \Omega, \\
		\\
		\dfrac{\partial f}{\partial n} = 0, & (x,y)\in \partial\Omega
	\end{cases}
\end{equation}
where $\Omega = (x_0, x_{M_1}) \times (y_0, y_{M_2})$, $A A^\top = \begin{pmatrix} \sigma_1^2 & \rho \sigma_1 \sigma_2 \\ \rho \sigma_1 \sigma_2 & \sigma_2^2 \end{pmatrix}$, $B = (b_1, b_2)$, and $\dfrac{\partial f}{\partial n}$ is normal derivative on the boundary. $\rho, \sigma_1, \sigma_2, r, b_1, b_2$ are continuous functions of $(x,y)$ satisfying
$$
\inf_{(x,y) \in \Omega} \sigma_1^2 > 0, \quad \inf_{(x,y) \in \Omega} \sigma_2^2 > 0, \quad \rho \in [-1, 1], \quad r > r_0 > 0 \quad \forall (x,y) \in \Omega,
$$
where $r_0$ is a positive constant.

we choose the parametrization:
$$
A = \begin{pmatrix} \sigma_1 \cos \theta & \sigma_1 \sin \theta \\ \sigma_2 \sin \theta & \sigma_2 \cos \theta \end{pmatrix},
\quad \theta = \frac{\arcsin \rho}{2} \in \left[ -\frac{\pi}{4}, \frac{\pi}{4} \right].
$$
The differential operators are defined as $D^2 f = \begin{pmatrix} \partial_{xx} f & \partial_{xy} f \\ \partial_{yx} f & \partial_{yy} f \end{pmatrix}$ and $\nabla f = \begin{pmatrix} \partial_x f \\ \partial_y f \end{pmatrix}$.\\

Unlike first-order boundary conditions, the homogeneous Neumann boundary condition corresponds to a reflecting boundary in the associated stochastic process. When the process reaches the boundary, it reflects back into the interior domain, continuing its trajectory without termination or boundary penalties.

For interior grid points $1 \le i \le M_1-1$, $1 \le j \le M_2-1$, the numerical scheme is:

\begin{equation} 
	\label{eq:neumann-interior-update}
	f_{h,i,j} = \frac{1}{4(1+r_{i,j}h)} \sum_{k=1}^{4}
	\mathcal{L}^{(1)} f_h\big(X_k^h,Y_k^h\big) + q(x_i, y_j) h,
\end{equation}
where \((X_k^h,Y_k^h)\) are the final branch positions obtained by applying the rules \eqref{eq:x-reflection},\eqref{eq:y-reflection} to the proposed branch points.

The four proposed branch points are defined as:
\begin{equation}
	\label{pr}
	\begin{aligned}
		(x^*_1,y^*_1) &= \big(x_i+b_{1,i,j}h+\alpha_{i,j}\sigma_{1,i,j}\sqrt{h},\;
		y_j+b_{2,i,j}h+\alpha_{i,j}\sigma_{2,i,j}\sqrt{h}\big),\\
		(x^*_2,y^*_2) &= \big(x_i+b_{1,i,j}h-\beta_{i,j}\sigma_{1,i,j}\sqrt{h},\;
		y_j+b_{2,i,j}h+\beta_{i,j}\sigma_{2,i,j}\sqrt{h}\big),\\
		(x^*_3,y^*_3) &= \big(x_i+b_{1,i,j}h-\alpha_{i,j}\sigma_{1,i,j}\sqrt{h},\;
		y_j+b_{2,i,j}h-\alpha_{i,j}\sigma_{2,i,j}\sqrt{h}\big),\\
		(x^*_4,y^*_4) &= \big(x_i+b_{1,i,j}h+\beta_{i,j}\sigma_{1,i,j}\sqrt{h},\;
		y_j+b_{2,i,j}h-\beta_{i,j}\sigma_{2,i,j}\sqrt{h}\big).
	\end{aligned}
\end{equation}
For each proposed branch point \((x_k^*, y_k^*)\), the final branch point \((X_k^h, Y_k^h)\) is computed according to the following rules:

\begin{equation} \label{eq:x-reflection}
	X_k^h = 
	\begin{cases}
		2x_0 - x_k^*, & x_k^* < x_0, \\
		2x_{M_1} - x_k^*, & x_k^* > x_{M_1}, \\
		x_k^*, & \text{otherwise},
	\end{cases}
\end{equation}

\begin{equation} \label{eq:y-reflection}
	Y_k^h = 
	\begin{cases}
		2y_0 - y_k^*, & y_k^* < y_0, \\
		2y_{M_2} - y_k^*, & y_k^* > y_{M_2}, \\
		y_k^*, & \text{otherwise}.
	\end{cases}
\end{equation}
For boundary points, we handle them as follows:
\[
\begin{aligned}
	f_{h,0,j} &= f_{h,1,j},      & 0 \le j \le M_2,\\
	f_{h,M_1,j} &= f_{h,M_1-1,j}, & 0 \le j \le M_2,\\
	f_{h,i,0} &= f_{h,i,1},      & 0 \le i \le M_1,\\
	f_{h,i,M_2} &= f_{h,i,M_2-1}, & 0 \le i \le M_1.
\end{aligned}
\]
\begin{algorithm}
	\caption{Implementation of Neumann Boundary Scheme \eqref{eq:neumann-interior-update}}
	\label{alg:neumann-scheme}
	\begin{algorithmic}[1]
		\Require Interior grid points $(x_i, y_j)$, $i=1,\dots,M_1-1$, $j=1,\dots,M_2-1$
		\Ensure Coefficient matrix $T$ and right-hand side vector $b$
		
		\State Initialize $T \gets \mathbf{0}_{(M_1+1)(M_2+1) \times (M_1+1)(M_2+1)}$
		\State Initialize $b \gets \mathbf{0}_{(M_1+1)(M_2+1)}$
		
		\For{$i = 1$ to $M_1-1$}
		\For{$j = 1$ to $M_2-1$}
		\State $idx \gets i(M_2+1) + j$ \Comment{Global index}
		\For{$k = 1$ to $4$}
		\State Compute proposed position $(x_k^*, y_k^*)$ via \eqref{pr}
		\State $(X_k^h, Y_k^h) \gets$ apply reflection rules \eqref{eq:x-reflection}, \eqref{eq:y-reflection}
		\State Find interpolation stencil for $\mathcal{L}^{(1)} f_h(X_k^h, Y_k^h)$
		\For{each grid point $(i_m, j_n)$ in stencil}
		\State $idx_{mn} \gets i_m (M_2+1) + j_n$
		\State $l_{mn} \gets$ interpolation weight
		\State $T(idx, idx_{mn}) \gets T(idx, idx_{mn}) - \dfrac{l_{mn}}{4(1 + r_{i,j} h)}$
		\EndFor
		\EndFor
		\State $T(idx, idx) \gets T(idx, idx) + 1$ \Comment{Diagonal entry}
		\State $b(idx) \gets b(idx) + q(x_i, y_j) h$ \Comment{Source term}
		\EndFor
		\EndFor
		
		\State At the boundary, the following conditions are applied for matrix assembly:
		\[
		\begin{aligned}
			f_{h,0,j} &= f_{h,1,j},      & 0\le j\le M_2,\\
			f_{h,M_1,j} &= f_{h,M_1-1,j}, & 0\le j\le M_2,\\
			f_{h,i,0} &= f_{h,i,1},      & 0\le i\le M_1,\\
			f_{h,i,M_2} &= f_{h,i,M_2-1}, & 0\le i\le M_1.
		\end{aligned}
		\]
		\State Solve linear system $T F = b$ for $F$ \Comment{$F$ contains all grid values}
	\end{algorithmic}
\end{algorithm}
\newpage
This is an implicit scheme. Given that $r(x,y) > r_0 > 0$ for all $(x,y) \in \Omega$, the coefficient matrix $T$ is an M-matrix, ensuring the positivity-preserving property. Since $A$ is a sparse M-matrix, we can employ iterative methods such as Jacobi, Gauss-Seidel, or successive over-relaxation (SOR) to solve the linear system.
\newpage
\begin{figure}[htbp]
	\centering
	\includegraphics[width=0.8\textwidth]{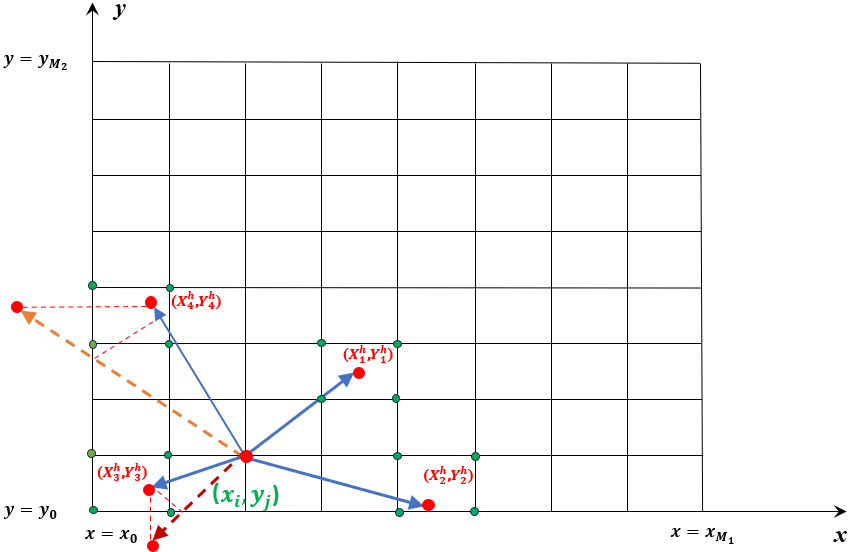}
	\caption{Illustration of Algorithm \ref{alg:neumann-scheme}. Here, $X_4^h < x_0$ and $Y_3^h < y_0$ are reflected back into the domain.}
	\label{fig2}
\end{figure}

\begin{theorem}
	\label{thm:neumann-convergence}
	Assume that the exact solution $f$ is sufficiently smooth and $h = h_1 = h_2$. Then the scheme \eqref{eq:neumann-interior-update} satisfies
	\[
	\|e\|_\infty \le O(h^{1/2}),
	\]
	where $e_{i,j}=f_{h,i,j}-f(x_i,y_j)$.
\end{theorem}

\begin{proof}
	We split the proof into three parts: interior nodes, boundary nodes, and the final recursion.
	
	\textbf{1. Interior nodes.}
	Consider interior nodes $i=1,\dots,M_1-1$ and $j=1,\dots,M_2-1$. To illustrate the reflection treatment, take the first branch as an example. Suppose the $x$--component of this branch crosses the vertical boundary $x=x_{M_1}$ and is reflected, while the $y$--component stays inside the domain. Then
	\[
	X_1^h = 2x_{M_1}-x_i - b_{1,i,j}h - \alpha_{i,j}\sigma_{1,i,j}\sqrt{h},\qquad
	Y_1^h = y_j + b_{2,i,j}h + \alpha_{i,j}\sigma_{2,i,j}\sqrt{h}.
	\]
	The other branches are treated similarly (reflected or not), and their Taylor expansions follow the same procedure.
	
	From the scheme \eqref{eq:neumann-interior-update} we obtain
	\[
	|e_{i,j}| \le |I_1| + \sum_{k=1}^4 \frac{1}{4(1+r_{i,j}h)}\bigl(\mathcal{L}^{(1)}f-f\bigr)(X_k^h,Y_k^h) + \frac{\|e\|_\infty}{1+r_0h} ,
	\]
	where 
	\[
	I_1 = \frac{1}{4(1+r_{i,j}h)}\sum_{k=1}^4 f(X_k^h,Y_k^h) - f(x_i,y_j) + q(x_i,y_j)h.
	\]
	
	Expand $f(X_1^h,Y_1^h)$ in a Taylor series:
	\[
	f(X_1^h,Y_1^h) = f_{i,j} + \sum_{p=1}^4 K_p + O(h^{3/2}),
	\]
	with
	\[
	\begin{aligned}
		K_1 &= 2(x_{M_1}-x_i)\Bigl[\partial_x f + \partial_{xy}f\bigl(b_{2,i,j}h+\alpha_{i,j}\sigma_{2,i,j}\sqrt{h}\bigr) + (x_{M_1}-x_i)\partial_{xx}f\Bigr]_{x_i,y_j},\\[2mm]
		K_2 &= \alpha_{i,j}\sigma_{1,i,j}\sqrt{h}\Bigl[-2\partial_{xx}f(x_{M_1}-x_i) - \partial_x f - \partial_{xy}f\bigl(b_{2,i,j}h+\alpha_{i,j}\sigma_{2,i,j}\sqrt{h}\bigr)\Bigr]_{x_i,y_j},\\[2mm]
		K_3 &= b_{1,i,j}h\Bigl[-2\partial_{xx}f(x_{M_1}-x_i) - \partial_x f\Bigr]_{x_i,y_j} + \partial_y f_{x_i,y_j}\bigl(b_{2,i,j}h+\alpha_{i,j}\sigma_{2,i,j}\sqrt{h}\bigr),\\[2mm]
		K_4 &= \frac12\partial_{xx}f_{x_i,y_j}(\sigma_{1,i,j})^2(\alpha_{i,j})^2h + \frac12\partial_{yy}f_{x_i,y_j}(\sigma_{2,i,j})^2(\alpha_{i,j})^2h.
	\end{aligned}
	\]
	
	Because $x_i + b_{1,i,j}h + \alpha_{i,j}\sigma_{1,i,j}\sqrt{h} \ge x_{M_1}$, we have $x_{M_1}-x_i = O(\sqrt{h})$. Using the Neumann boundary condition $\partial f/\partial n = 0$ on $\partial\Omega$, we obtain the key relations
	\[
	\begin{aligned}
		&\Bigl(\partial_x f + \partial_{xy}f\bigl(b_{2,i,j}h+\alpha_{i,j}\sigma_{2,i,j}\sqrt{h}\bigr) + \partial_{xx}f(x_{M_1}-x_i)\Bigr)_{x_i,y_j} + O(h^{1/2}) = 0,\\
		&-\partial_{xx}f_{x_i,y_j}(x_{M_1}-x_i) = \partial_x f_{x_i,y_j} + \partial_{xy}f_{x_i,y_j}\bigl(b_{2,i,j}h+\alpha_{i,j}\sigma_{2,i,j}\sqrt{h}\bigr) + O(h),\\
		&-\partial_{xx}f_{x_i,y_j}(x_{M_1}-x_i) = \partial_x f_{x_i,y_j} + O(h).
	\end{aligned}
	\]
	
	Substituting these relations into the expansion yields
	\[
	\begin{aligned}
		f(X_1^h,Y_1^h) &= f_{i,j} + \partial_x f_{x_i,y_j}\bigl(b_{1,i,j}h+\alpha_{i,j}\sigma_{1,i,j}\sqrt{h}\bigr)\\
		&\quad + \partial_y f_{x_i,y_j}\bigl(b_{2,i,j}h+\alpha_{i,j}\sigma_{2,i,j}\sqrt{h}\bigr)\\
		&\quad + \partial_{xy}f_{x_i,y_j}(\alpha_{i,j})^2\sigma_{1,i,j}\sigma_{2,i,j}h\\
		&\quad + \frac12\partial_{xx}f_{x_i,y_j}(\sigma_{1,i,j})^2(\alpha_{i,j})^2h + \frac12\partial_{yy}f_{x_i,y_j}(\sigma_{2,i,j})^2(\alpha_{i,j})^2h + O(h^{3/2}).
	\end{aligned}
	\]
	
	Analogous expansions hold for $f(X_2^h,Y_2^h)$, $f(X_3^h,Y_3^h)$, and $f(X_4^h,Y_4^h)$. Combining them and using the governing equation \eqref{eq:elliptic-neumann} gives $|I_1| = O(h^{3/2})$. Consequently, for interior points,
	\[
	|e_{i,j}| \le \frac{\|e\|_\infty}{1+r_0h} + O(h^{3/2}).
	\]
	
	\textbf{2. Boundary nodes.}
	Take the right boundary $i=M_1$ and $j=1,\dots,M_2-1$ as an example. By the discrete Neumann condition we have $f_{h,M_1,j}=f_{h,M_1-1,j}$, hence
	\[
	|e_{M_1,j}| \le |f_{h,M_1-1,j}-f(x_{M_1-1},y_j)| + |f(x_{M_1},y_j)-f(x_{M_1-1},y_j)|.
	\]
	Using the continuous Neumann condition and a Taylor expansion,
	\[
	f(x_{M_1-1},y_j) = f(x_{M_1},y_j) + \partial_x f(x_{M_1},y_j)(x_{M_1-1}-x_{M_1}) + O(h^2) = f(x_{M_1},y_j) + O(h^2).
	\]
	Thus
	\[
	|e_{M_1,j}| \le \frac{\|e\|_\infty}{1+r_0h} + O(h^{3/2}).
	\]
	Similar estimates hold for the other boundary segments.
	
	\textbf{3. Global error estimate.}
	Collecting the estimates for all nodes, we have
	\[
	|e_{i,j}| \le \frac{\|e\|_\infty}{1+r_0h} + C h^{3/2}
	\]
	for some constant $C$ independent of $h$. Taking the maximum over all $(i,j)$ gives
	\[
	\|e\|_\infty \le \frac{\|e\|_\infty}{1+r_0h} + C h^{3/2}.
	\]
	Rearranging,
	\[
	\left(1-\frac{1}{1+r_0h}\right)\|e\|_\infty \le C h^{3/2}
	\quad\Longrightarrow\quad
	\frac{r_0h}{1+r_0h}\|e\|_\infty \le C h^{3/2},
	\]
	so that
	\[
	\|e\|_\infty \le C\frac{1+r_0h}{r_0}\,h^{1/2} = O(h^{1/2}).
	\]
	This completes the proof.
\end{proof}

\subsection{Periodic Boundary Conditions}
\label{subsec:periodic}

We now consider periodic boundary conditions.
Let $\Omega = (x_0,x_{M_1}) \times (y_0,y_{M_2})$ denote the computational domain, and define the spatial periods
\[
L_x := x_{M_1}-x_0, \qquad L_y := y_{M_2}-y_0.
\]
The solution $f$ and all PDE coefficients are assumed to be defined on $\mathbb{R}^2$ and to be doubly periodic with respect to these periods.  
That is, for all $(x,y)\in\mathbb{R}^2$ and all integers $m,n$,
\[
f(x + mL_x,\, y) = f(x,y), 
\qquad
f(x,\, y + nL_y) = f(x,y),
\]
and the same periodicity holds for $A$, $B$, and $r$ in \eqref{eq:PDE-periodic}.  

Under periodicity, the PDE is naturally posed on the 2D torus 
$\mathbb{T}^2 := \mathbb{R}^2 / (L_x\mathbb{Z} \times L_y\mathbb{Z})$:

\begin{equation}  \label{eq:PDE-periodic}
	\frac{1}{2}\mathrm{Tr}(AA^\top D^2 f) + B\cdot\nabla f - r(x,y)f = 0, 
	(x,y)\in\mathbb{R}^2, \\
\end{equation}
where the diffusion tensor and convection vector are given by
\[
AA^\top = \begin{pmatrix} \sigma_1^2 & \rho \sigma_1 \sigma_2 \\ \rho \sigma_1 \sigma_2 & \sigma_2^2 \end{pmatrix}, \quad 
B = (b_1, b_2)^\top,
\]
with coefficients satisfying the regularity conditions:
\[
\inf_{(x,y) \in \Omega} \sigma_1^2 > 0, \quad \inf_{(x,y) \in \Omega} \sigma_2^2 > 0, \quad 
\rho \in [-1, 1], \quad r > r_0 > 0.
\]

We employ the parametrization:
\[
A = \begin{pmatrix} \sigma_1 \cos \theta & \sigma_1 \sin \theta \\ \sigma_2 \sin \theta & \sigma_2 \cos \theta \end{pmatrix},
\quad \theta = \frac{\arcsin \rho}{2} \in \left[ -\frac{\pi}{4}, \frac{\pi}{4} \right].
\]

In the discrete approximation, any branch endpoint that leaves $\Omega$ must be wrapped 
back into the fundamental cell via periodicity.  
Define the wrapping operator $\mathcal{W}: \mathbb{R}^2\to\Omega$ by
\begin{equation}
	\label{wrap}
	\mathcal{W}(x,y)
	=
	\bigl(x_0 + \{(x-x_0)\bmod L_x\},\ 
	y_0 + \{(y-y_0)\bmod L_y\}\bigr),
\end{equation}
For each grid node $(x_i,y_j,t_n)$, compute the four  endpoints by \eqref{pr}
\[
(x_k^*, y_k^*), \qquad k=1,\dots,4,
\]
and then apply periodic wrapping:
\[
(X_k^h,Y_k^h) = \mathcal{W}(x_k^*, y_k^*), \qquad k=1,\dots,4.
\]

With no boundary-induced stopping, all branches evolve over the full step \(h\), yielding the explicit update:
\begin{equation} \label{eq:periodic-update}
	f_{h,i,j} = \frac{1}{4(1 + r_{i,j} h)} \sum_{k=1}^4 \mathcal{L}^{(1)} f_h(X_k^h,Y_k^h)+ q(x_i, y_j) h.
\end{equation}

\begin{algorithm}
	\caption{Implementation of Periodic Boundary Scheme \eqref{eq:periodic-update}}
	\label{alg:periodic}
	\begin{algorithmic}[1]
		\Require Grid points $(x_i, y_j)$ for $i=0,\dots,M_1$, $j=0,\dots,M_2$
		\Ensure Numerical solution $f_{h,i,j}$
		
		\State Initialize $T \gets \mathbf{0}_{(M_1+1)(M_2+1) \times (M_1+1)(M_2+1)}$
		\State Initialize $b \gets \mathbf{0}_{(M_1+1)(M_2+1)}$
		
		\For{$i = 0$ to $M_1$}
		\For{$j = 0$ to $M_2$}
		\State $idx \gets i(M_2+1) + j$
		\For{$k = 1$ to $4$}
		\State Compute proposed position $(x_k^*, y_k^*)$ via \eqref{pr}
		\State $(X_k^h, Y_k^h) \gets \mathcal{W}(x_k^*, y_k^*)$ \Comment{using \eqref{wrap}}
		\State Find interpolation stencil for $\mathcal{L}^{(1)} f_h(X_k^h, Y_k^h)$
		\For{each grid point $(i_m, j_n)$ in stencil}
		\State $idx_{mn} \gets i_m (M_2+1) + j_n$
		\State $l_{mn} \gets$ interpolation weight
		\State $T(idx, idx_{mn}) \gets T(idx, idx_{mn}) - \dfrac{l_{mn}}{4(1 + r_{i,j} h)}$
		\EndFor
		\EndFor
		\State $T(idx, idx) \gets T(idx, idx) + 1$ \Comment{Diagonal entry}
		\State $b(idx) \gets b(idx) + q(x_i, y_j) h$ \Comment{Source term from \eqref{eq:periodic-update}}
		\EndFor
		\EndFor
		
		\State Solve linear system $T F = b$ for $F$
		\State $f_{h,i,j} \gets$ extract values from $F$
	\end{algorithmic}
\end{algorithm}
This is an implicit scheme. Given that $r(x,y) > r_0 > 0$ for all $(x,y) \in \Omega$, the coefficient matrix $T$ is an M-matrix, ensuring the positivity-preserving property. Since $A$ is a sparse M-matrix, we can employ iterative methods such as Jacobi, Gauss-Seidel, or successive over-relaxation (SOR) to solve the linear system.
\begin{figure}[H]
	\centering
	\includegraphics[width=0.8\textwidth]{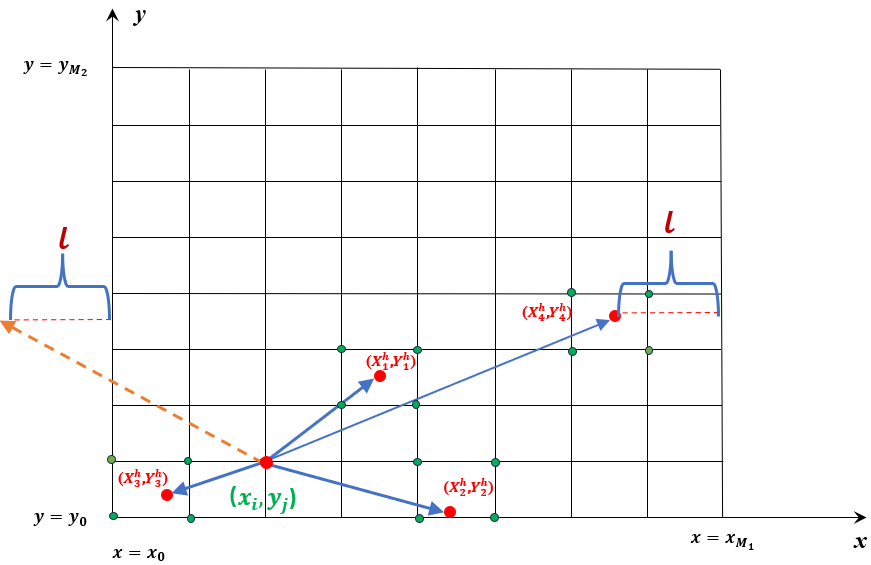}
	\caption{Schematic illustration of Algorithm \ref{alg:periodic}, showing periodic wrapping of branch endpoints.}
	\label{fig:periodic-scheme}
\end{figure}

\begin{theorem}[Improved Convergence under Periodicity]
	\label{thm:periodic-convergence}
	Assume the exact solution \(f\) is smooth and spatially periodic on \(\Omega\). With \(h = h_1 = h_2\), there exists a constant \(C > 0\), independent of mesh sizes, such that the numerical solution from Algorithm \ref{alg:periodic} satisfies
	\[
	\|e\|_\infty \le Ch.
	\]
\end{theorem}

\begin{proof}[Proof sketch]
	Under periodic boundary conditions, all spatial boundary hitting parameters equal \(h\), eliminating asymmetric stopping effects. Taylor expansion of the four branch contributions around \((x_i, y_j)\) shows exact cancellation of odd spatial moments and proper matching of second-order moments with diffusion coefficients, yielding quadrature error \(O(h^2)\). Combined with spatial interpolation error \(O(h^2)\), the local truncation error becomes:
	\[
	R_{\mathrm{loc}} = O(h^2).
	\]
	Global error scales as \(R_{\mathrm{loc}}/h\), giving first-order convergence \(O(h)\).
\end{proof}

\section{Numerical Results}
\label{sec:numerical-results}

In this section, we present numerical experiments to verify the convergence behavior of the schemes proposed in Section~\ref{sec2}. We report both the maximum-norm errors and the $L^2$-norm errors:

\[
\text{Error}_{L^\infty} = \max_{0 \le i \le M_1} \max_{0 \le j \le M_2} \big| f_{h,i,j} - f(x_i, y_j) \big|,
\]
\[
\text{Error}_{L^2} = \left( h_1 h_2 \sum_{i=0}^{M_1} \sum_{j=0}^{M_2} \big| f_{h,i,j} - f(x_i, y_j) \big|^2 \right)^{1/2}.
\]

The empirical convergence rates are computed between successive refinements by
\[
\text{Rate} =\dfrac{\log\left( {\text{Error}(h_1)}/{\text{Error}(h_2)} \right)}{\log {(h_1/h_2)}},
\]
where $\text{Error}(h)$ denotes the error at grid spacing $h$.

All runs use the scaling $h = h_1 = h_2$ , unless stated otherwise.

We begin with Dirichlet boundary conditions and compare  Algorithm \ref{alg:scheme-elliptic-dirichlet}  with the LISL method described in \cite{bib4}, using two boundary treatments: Exact and Extrapolation. We subsequently examine performance under homogeneous Neumann and periodic boundary conditions.

\subsection{Dirichlet Boundary Conditions}

Consider the following equation:

\begin{equation}
	\left\{
	\begin{aligned}
		& \frac{\sigma_1^2}{2} \frac{\partial^2 f}{\partial x^2} + \frac{\sigma_2^2}{2} \frac{\partial^2 f}{\partial y^2} + \rho \sigma_1 \sigma_2 \frac{\partial^2 f}{\partial x \partial y} - r f + q = 0, \quad (x, y) \in (0, 1) \times (0, 1), \\
		& f(1,y) =e^{1+y}, \quad y \in [0,1], \\
		& f(0,y) =  e^{y}, \quad y \in [0,1], \\
		& f(x,1) =  e^{1+x}, \quad x \in [0,1], \\
		& f(x,0) = e^{x}, \quad x \in [0,1].
	\end{aligned}
	\right.
\end{equation}

The exact solution is $f(x,y) =  e^{x+y}$.
where $\sigma_1 = \dfrac{x+1}{2}$, $\sigma_2 = {x+1}$, $\rho = 0.9$, $r = 1.075\cdot(x+1)^2$,$q=0.$
The associated covariance matrix is
\[
A A^\top = (x+1)^2
\begin{pmatrix}
	\dfrac{1}{4} & \dfrac{9}{20} \\[2pt]
	\\
	\dfrac{9}{20} & 1
\end{pmatrix},
\]
and the inner \(2 \times 2\) coefficient matrix is {not diagonally dominant}. Thus, this periodic test likewise represents an anisotropic mixed-derivative diffusion case.

We compare Algorithms~\ref{alg:scheme-elliptic-dirichlet} with the  the semi-Lagrangian (SL) method{\cite{bib4}}. Using the fifth type of approximation from section 5.1 of the original text.

\textbf{Boundary treatments.}
Due to stencil non-compactness, LISL evaluation points may fall outside $\Omega$. We test:

\begin{itemize}
	\item \textbf{Exact}:  
	Outside points are assigned the exact value $e^{x+y}$.  
	This can achieve the theoretical convergence order but is not practical.
	
	\item \textbf{Extrapolation}:  
	For a point $(x_q, y_q)$ outside $\Omega$, we first compute  indices:
	$i_0 = (x_q - x_{0})/h_1 + 1$ and $j_0 = (y_q - y_{0})/h_2 + 1$. 
	Then we clamp the indices to the nearest boundary cell by setting 
	$i = \min(\max(\lfloor i_0 \rfloor, 1), M_1-1)$ and $j = \min(\max(\lfloor j_0 \rfloor, 1), M_2-1)$ interpolation at $(x_q, y_q)$ using the four grid nodes 
	$(x_i, y_j), (x_{i+1}, y_j), (x_i, y_{j+1}), (x_{i+1}, y_{j+1})$. 
\end{itemize}

\subsubsection{Numerical Results}
\begin{table}[!h]
	\centering
	\caption{$L^\infty$ and $L^2$ Errors with Convergence Rates (Algorithm \ref{alg:scheme-elliptic-dirichlet})}
	\label{tab:dirichlet-errors}
	\small
	\begin{tabular}{cccccc}
		\toprule
		$M_1$ (=1/$h_1$) & $M_2$ (=1/$h_2$) & $\text{Error}_{L^\infty}$ & $\text{Rate}_{L^\infty}$ & $\text{Error}_{L^2}$ & $\text{Rate}_{L^2}$ \\
		\midrule
		20 & 20 & 8.5027e-2 & --- & 2.9839e-2 & --- \\
		40 & 40 & 4.0856e-2 & 1.0574 & 1.4857e-2 & 1.0060 \\
		80 & 80 & 1.8062e-2 & 1.1776 & 7.0148e-3 & 1.0827 \\
		160 & 160 & 7.6846e-3 & 1.2329 & 3.2701e-3 & 1.1011 \\
		320 & 320 & 3.4316e-3 & 1.1631 & 1.5220e-3 & 1.1034 \\
		640 & 640 & 1.5755e-3 & 1.1230 & 7.1417e-4 & 1.0916 \\
		\bottomrule
	\end{tabular}
\end{table}
\begin{table}[!h]
	\centering
	\caption{LISL-Exact scheme: $L^\infty$ errors at $t=0$ and convergence rates for the Dirichlet problem.}
	\label{tab:lisl_exact}
	\small
	\begin{tabular}{cccccc}
		\toprule
		$M_1$ (=1/$h_1$) & $M_2$ (=1/$h_2$) & $\text{Error}_{L^\infty}$ & $\text{Rate}_{L^\infty}$ & $\text{Error}_{L^2}$ & $\text{Rate}_{L^2}$ \\
		\midrule
		20 & 20 & 1.9640e-02 & --- & 9.9800e-03 & --- \\
		40 & 40 & 9.2600e-03 & 1.084 & 4.6710e-03 & 1.095 \\
		80 & 80 & 4.2740e-03 & 1.115 & 2.2020e-03 & 1.085 \\
		160 & 160 & 2.0070e-03 & 1.091 & 1.0430e-03 & 1.078 \\
		320 & 320 & 9.6830e-04 & 1.052 & 4.9950e-04 & 1.062 \\
		640 & 640 & 4.6930e-04 & 1.045 & 2.4120e-04 & 1.050 \\
		\bottomrule
	\end{tabular}
\end{table}
\begin{table}[!h]
	\centering
	\caption{LISL-Extrapolation scheme: $L^\infty$ errors at $t=0$ and convergence rates for the Dirichlet problem. }
	\label{tab:lisl_exact}
	\small
	\begin{tabular}{cccccc}
		\toprule
		$M_1$ (=1/$h_1$) & $M_2$ (=1/$h_2$) & $\text{Error}_{L^\infty}$ & $\text{Rate}_{L^\infty}$ & $\text{Error}_{L^2}$ & $\text{Rate}_{L^2}$ \\
		\midrule
		20 & 20 & 1.5770e+00 & --- & 2.8130e-01 & --- \\
		40 & 40 & 1.2960e+00 & 0.283 & 2.0680e-01 & 0.443 \\
		80 & 80 & 9.9900e-01 & 0.375 & 1.4760e-01 & 0.487 \\
		160 & 160 & 7.4520e-01 & 0.422 & 1.0450e-01 & 0.499 \\
		320 & 320 & 5.4900e-01 & 0.441 & 7.3740e-02 & 0.503 \\
		640 & 640 & 3.9380e-01 & 0.479 & 5.2010e-02 & 0.504 \\
		\bottomrule
	\end{tabular}
\end{table}
\subsubsection{Discussion}

The results presented in Tables~\ref{tab:dirichlet-errors}--\ref{tab:lisl_exact} demonstrate the following:

\textbf{1. Convergence performance.}
The numerical experiments show that the convergence rate of Algorithm \ref{alg:scheme-elliptic-dirichlet} is consistent with the Theorem~\ref{thm:convergence}.

In comparison, the LISL--Exact method also achieves near first-order accuracy, which is consistent with theoretical expectations. However, it relies on knowing the exact solution outside the domain and is therefore impractical. The LISL--Extrapolation variant, which uses bilinear extrapolation at the boundary, performs markedly worse, with convergence rates around $0.4$---clearly illustrating how  extrapolation  degrades the accuracy of semi-Lagrangian schemes.

\textbf{2. Boundary treatment.}
Algorithms~\ref{alg:scheme-elliptic-dirichlet}  naturally incorporate boundary interactions through their stopping time mechanism and the adaptive probabilities defined in \eqref{eq:spatial-weights}. Consequently, they require neither extrapolation nor exact exterior values, and they maintain  accuracy and monotonicity  up to the boundary.

By contrast, the semi-Lagrangian method faces an inherent trade-off: it either depends on impractical exact boundary data or suffers from poor convergence when extrapolation is employed. This limitation underscores a structural difficulty of LISL-type schemes in handling boundaries accurately.
\medskip

In summary, the expectation-based schemes proposed here display much more robust boundary handling than the semi-Lagrangian approach. By avoiding both extrapolation and the need for exact boundary information, they offer a more reliable and practical choice for anisotropic diffusion problems.
\subsection{ Homogeneous Neumann Boundary Conditions}

Consider the following equation:

\begin{equation}
	\left\{
	\begin{aligned}
		& \frac{\sigma_1^2}{2} \frac{\partial^2 f}{\partial x^2} + \frac{\sigma_2^2}{2} \frac{\partial^2 f}{\partial y^2} + \rho \sigma_1 \sigma_2 \frac{\partial^2 f}{\partial x \partial y} - r f + q = 0, \quad (x, y) \in (0, 1) \times (0, 1), \\
		&\dfrac{\partial f}{\partial n}=0 ,(x,y)\in\partial\Omega.
	\end{aligned}
	\right.
\end{equation}

where $\sigma_1 = \dfrac{x+1}{2\pi}$, $\sigma_2 = \dfrac{x+1}{\pi}$, $\rho = 0.9$, $r =4$, and
\[\begin{aligned}
	q &= -0.45(x+1)^2\cos(\pi x - \dfrac{\pi}{2})\cos(\pi y - \dfrac{\pi}{2}) +\dfrac{5}{8} (x+1)^2\sin(\pi x - \dfrac{\pi}{2})\sin(\pi y - \dfrac{\pi}{2})\\&+4\sin(\pi x - \dfrac{\pi}{2})\sin(\pi y - \dfrac{\pi}{2}).
\end{aligned}\]
The associated covariance matrix is
\[
A A^\top = (x+1)^2
\begin{pmatrix}
	\dfrac{1}{4\pi^2} & \dfrac{9}{20\pi^2} \\[2pt]
	\\
	\dfrac{9}{20\pi^2} & \dfrac{1}{\pi^2}
\end{pmatrix},
\]
and the inner \(2 \times 2\) coefficient matrix is {not diagonally dominant}. Thus, this periodic test likewise represents an anisotropic mixed-derivative diffusion case.

The exact solution is $f(x,y) = \sin{(\pi x - \dfrac{\pi}{2})}\sin{(\pi y - \dfrac{\pi}{2})}$.

\begin{table}[!h]
	\centering
	\caption{$L^\infty$ and $L^2$ Errors with Convergence Rates (Algorithm \ref{alg:neumann-scheme})}
	\label{tab:neumann-results}
	\small
	\begin{tabular}{cccccc}
		\toprule
		$M_1$ (=1/$h_1$) & $M_2$ (=1/$h_2$) & $L^\infty$ Error & $L^2$ Error & $L^\infty$ Rate & $L^2$ Rate \\
		\midrule
		20 & 20 & 1.9990e-1 & 9.9898e-2 & --- & --- \\
		40 & 40 & 1.0111e-1 & 4.9244e-2 & 0.9834 & 1.0205 \\
		80 & 80 & 5.1879e-2 & 2.4765e-2 & 0.9627 & 0.9916 \\
		160 & 160 & 2.5990e-2 & 1.2261e-2 & 0.9972 & 1.0142 \\
		320 & 320 & 1.2961e-2 & 6.1752e-3 & 1.0038 & 0.9895 \\
		640 & 640 & 6.5744e-3 & 3.0489e-3 & 0.9792 & 1.0182 \\
		\bottomrule
	\end{tabular}
\end{table}
The observed rates for the reflective test exceed the conservative theoretical bound in some grid regimes; this is consistent with the fact that for most interior nodes the four branches do not hit the boundary and therefore attain higher local accuracy.
\subsection{Periodic Boundary Conditions}
\label{subsec:numerics-periodic}

The exact solution  
\[
f(x, y) =  \sin(2\pi x) \sin(2\pi y)
\]  
is defined on $\mathbb{R}^2$, with the computational domain taken as the unit square \([0, 1] \times [0, 1]\). Periodic boundary conditions (with periods \(L_x = L_y = 1\)) are applied, 
which satisfies the backward PDE:
\begin{equation} 
	\label{eq:exp-periodic}
	\displaystyle
	\frac{1}{2} \sigma_1^2 \partial_{xx} f + \frac{1}{2} \sigma_2^2 \partial_{yy} f + \rho \sigma_1 \sigma_2 \partial_{xy} f + b_1 \partial_x f + b_2 \partial_y f - r f = 0.
\end{equation}
with periodic boundary conditions and coefficients defined as:
\[
\begin{aligned}
	&b_1 = \sin(2\pi x) \cos(2\pi y), \\
	&b_2 = -\sin(2\pi y) \cos(2\pi x), \\
	&\sigma_1 = \frac{1}{2\pi} \left(\sin(2\pi x) \sin(2\pi y)+2\right), \\
	&\sigma_2 =  \frac{1}{4\pi} \left(\sin(2\pi x) \sin(2\pi y)+2\right), \\
	&\rho = 0.9, \\
	&r(x,y) =4.\\
	&q = \frac{5}{8} \left( \sin(2\pi x ) \sin(2\pi y ) + 2 \right)^2 \sin(2\pi x) \sin(2\pi y) - \frac{9}{20} \left( \sin(2\pi x) \sin(2\pi y ) + 2 \right)^2 \cos(2\pi x ) \cos(2\pi y ) \\&+ 4 \sin(2\pi x ) \sin(2\pi y )
\end{aligned}
\]
The associated covariance matrix is
\[
A A^\top =  \frac{1}{4\pi^2} \left(\sin(2\pi x) \sin(2\pi y)+2\right)^2
\begin{pmatrix}
	1 & 0.45 \\[2pt]
	\\
	0.45 & 0.25
\end{pmatrix},
\]
and the inner \(2 \times 2\) coefficient matrix is {not diagonally dominant}. Thus, this periodic test likewise represents an anisotropic mixed-derivative diffusion case.

The periodic run errors at \(t = 0\) are given in Table~\ref{tab:periodic_consistent}.

\begin{table}[htb]
	\centering
	\caption{Maximum absolute error and $L^2$ error with convergence rates at $t=0$ for the periodic problem (Algorithm~\ref{alg:periodic}).}
	\label{tab:periodic_consistent}
	\small
	\begin{tabular}{@{}ccccccc@{}}
		\toprule
		\(M_1\) & \(M_2\) & \(N\) &$\text{Error}_{L^\infty}$ & $\text{Rate}_{L^\infty}$ & $\text{Error}_{L^2}$ & $\text{Rate}_{L^2}$ \\
		\midrule
		20  & 20  & 20  & \(2.0575 \times 10^{-1}\) & --- & \(7.7985 \times 10^{-2}\) & --- \\
		40  & 40  & 40  & \(9.7576 \times 10^{-2}\) & 1.08 & \(3.6975 \times 10^{-2}\) & 1.08 \\
		80  & 80  & 80  & \(4.6353 \times 10^{-2}\) & 1.07 & \(1.8095 \times 10^{-2}\) & 1.03 \\
		160 & 160 & 160 & \(2.2804 \times 10^{-2}\) & 1.02 & \(8.8716 \times 10^{-3}\) & 1.03 \\
		320 & 320 & 320 & \(1.1143 \times 10^{-2}\) & 1.03 & \(4.4519 \times 10^{-3}\) & 0.99 \\
		640 & 640 & 640 & \(5.5619 \times 10^{-3}\) & 1.00 & \(2.2124 \times 10^{-3}\) & 1.01 \\
		\bottomrule
	\end{tabular}
\end{table}

The periodic experiment exhibits empirical convergence close to first order on refined meshes, in agreement with Theorem~\ref{thm:periodic-convergence}.

\section{Conclusion}
\label{sec4}
This paper has introduced a novel non-compact positivity-preserving numerical framework for elliptic partial differential equations, built upon the Feynman-Kac formula and mathematical expectation. By representing the solution as a conditional expectation of a discretized stochastic process and approximating it through finite branching paths with positivity-preserving interpolation, we obtain schemes that inherently preserve non-negativity while handling anisotropic diffusion without diagonal dominance constraints.

A key contribution lies in the systematic development of numerical schemes for different boundary conditions. For first-type (Dirichlet) boundaries, we compute stopping probabilities and payoffs at branch endpoints with adaptive treatment, achieving first-order convergence in both $L^\infty$ and $L^2$ norms. For second-type (homogeneous Neumann) boundaries, boundary contacts are modeled as reflecting events within the stochastic process, yielding theoretical convergence of order $O(h^{1/2})$ with empirically observed higher rates in practice. The framework naturally extends to periodic boundaries through modular wrapping, restoring first-order accuracy.

Through rigorous theoretical analysis, we have established the consistency, stability, and positivity-preserving properties of the proposed schemes. The implicit formulation allows for large time steps while maintaining numerical stability, and the non-compact stencil enables handling of general anisotropic problems. Numerical experiments confirm the theoretical convergence rates and demonstrate the effectiveness of the approach across all boundary types.

This research contributes a cohesive numerical methodology that synthesizes three important features:
\begin{itemize}
	\item {Inherent positivity preservation} through constructive design
	\item {Physics-compatible boundary treatments} for diverse condition types  
	\item {Rigorous \(L^\infty\) convergence guarantees} under practical discretization
\end{itemize}
\section*{Acknowledgments}

This work was supported by the Natural Science Foundation of China (no. 12371401).

\bibliographystyle{amsplain}

\begin{thebibliography}{25}
	
	\bibitem{bib1}
	G. Barles and P. E. Souganidis.
	Convergence of approximation schemes for fully nonlinear second order equations.
	\emph{Asymptot. Anal.}, 4:271--283, 1991.
	
	\bibitem{bib2}
	J. F. Bonnans and H. Zidani.
	Consistency of generalized finite difference schemes for the stochastic HJB equation.
	\emph{SIAM J. Numer. Anal.}, 41:1008--1021, 2003.
	
	\bibitem{bib3}
	M. G. Crandall and P. L. Lions.
	Convergent difference schemes for nonlinear parabolic equations and mean curvature motion.
	\emph{Numer. Math.}, 75:17--41, 1996.
	
	\bibitem{bib4}
	K. Debrabant and E. R. Jakobsen.
	Semi-Lagrangian schemes for linear and fully non-linear diffusion equations.
	\emph{Math. Comp.}, 82:1433--1462, 2013.
	
	\bibitem{bib5}
	D. J. Duffy.
	\emph{Finite Difference Methods in Financial Engineering: A Partial Differential Equation Approach}.
	John Wiley and Sons, Chichester, 2006.
	
	\bibitem{bib6}
	A. Fahim, N. Touzi, and X. Warin.
	A probabilistic numerical method for fully nonlinear parabolic PDEs.
	\emph{Ann. Appl. Probab.}, 21:1322--1364, 2011.
	
	\bibitem{bib7}
	Z. M. Gao and J. M. Wu.
	A second-order positivity-preserving finite volume scheme for diffusion equations on general meshes.
	\emph{SIAM J. Sci. Comput.}, 37:A420--A438, 2015.
	
	\bibitem{bib8}
	W. J. Guo, J. F. Zhang, and J. Zhuo.
	A monotone scheme for high-dimensional fully nonlinear PDEs.
	\emph{Ann. Appl. Probab.}, 25:1540--1580, 2015.
	
	\bibitem{bib9}
	H. J. Kushner and P. G. Dupuis.
	\emph{Numerical Methods for Stochastic Control Problems in Continuous Time}.
	Springer-Verlag, New York, 1992.
	
	\bibitem{bib10}
	D. Kuzmin, M. J. Shashkov, and D. Svyatskiy.
	A constrained finite element method satisfying the discrete maximum principle for anisotropic diffusion problems.
	\emph{J. Comput. Phys.}, 228:3448--3463, 2009.
	
	\bibitem{bib11}
	T. S. Motzkin and W. Wasow.
	On the approximation of linear elliptic differential equations by difference equations with positive coefficients.
	\emph{J. Math. Phys.}, 31:253--259, 1952.
	
	\bibitem{bib12}
	J. M. Nordbotten, I. Aavatsmark, and G. T. Eigestad.
	Monotonicity of control volume methods.
	\emph{Numer. Math.}, 106:255--288, 2007.
	
	\bibitem{bib13}
	P. Sharma and G. W. Hammett.
	Preserving monotonicity in anisotropic diffusion.
	\emph{J. Comput. Phys.}, 227:123--142, 2007.
	
	\bibitem{bib14}
	Z. Q. Sheng and G. W. Yuan.
	A nine point scheme for the approximation of diffusion operators on distorted quadrilateral meshes.
	\emph{SIAM J. Sci. Comput.}, 30:1341--1361, 2008.
	
	\bibitem{bib16}
	S. H. Xu, X. F. Chen, C. Liu, and others.
	Numerical method for multi-alleles genetic drift problem.
	\emph{SIAM J. Numer. Anal.}, 57:1770--1788, 2019.
	
	\bibitem{bib17}
	G. W. Yuan and Z. Q. Sheng.
	Analysis of accuracy of a finite volume scheme for diffusion equations on distorted meshes.
	\emph{J. Comput. Phys.}, 224:1170--1189, 2007.
	
	\bibitem{bib18}
	G. W. Yuan and Z. Q. Sheng.
	Monotone finite volume schemes for diffusion equations on polygonal meshes.
	\emph{J. Comput. Phys.}, 227:6288--6312, 2008.
	
	\bibitem{bib20}
	Jing Li, Hongtao Yang, Yonghai Li, and Guangwei Yuan.
	The corrected finite volume element methods for diffusion equations satisfying discrete extremum principle.
	\emph{Commun. Comput. Phys.}, 32(5):1437--1473, 2022.
	
	\bibitem{bib21}
	Konstantin Lipnikov, Gianmarco Manzini, and Mikhail Shashkov.
	A vertex-centered and positivity-preserving scheme for anisotropic diffusion problems on arbitrary polygonal grids.
	\emph{J. Comput. Phys.}, 344:455--474, 2017.
	
	\bibitem{bib22}
	A. Barth, A. Lang, and C. Schwab.
	Multilevel Monte Carlo method for parabolic stochastic partial differential equations.
	\emph{BIT Numer. Math.}, 53:3--27, 2013.
	
	\bibitem{bib23}
	Jie Ren, Haoran Xu, and Xingye Yue.
	Expectation-based positivity-preserving noncompact numerical schemes.
	\emph{Sci. Sin. Math.}, 54(3):515--528, 2024.
	
	\bibitem{bib24}
	K. Ma and P. A. Forsyth.
	An unconditionally monotone numerical scheme for the two-factor uncertain volatility model.
	\emph{IMA J. Numer. Anal.}, 37(2):905--944, 2017.
	
	\bibitem{bib25}
	Haoran Xu, Jie Ren, and Xingye Yue.
	A non-compact positivity-preserving scheme for parabolic PDE via conditional expectation.
	arXiv preprint arXiv:2601.10977, 2026.
	
\end{thebibliography}

\end{document}